\documentclass[11pt, twoside, a4paper]{article}

\usepackage{amsmath}     				
\usepackage{amssymb}   					
\usepackage{amsthm}     				
\usepackage{graphicx}    				
\usepackage{textcomp}    				
\usepackage[T1]{fontenc} 				
\usepackage{marvosym}    				
\usepackage[sc]{mathpazo}  	 			
\usepackage{authblk} 					
\usepackage[usenames]{xcolor}  			
\usepackage{lastpage} 					
\usepackage{enumerate}	 				
\usepackage{hyperref} 					

\definecolor{ForestGreen}{rgb}{0.15,0.416,0.18}
\definecolor{EgyptBlue}{rgb}{0.063,0.2,0.65}
\hypersetup{
	colorlinks=true,
	linkcolor=EgyptBlue,
   citecolor=EgyptBlue,
   urlcolor=ForestGreen
}

\newtheorem{theorem}{Theorem}[section]
\newtheorem{corollary}[theorem]{Corollary}
\newtheorem{lemma}[theorem]{Lemma}
\newtheorem{proposition}[theorem]{Proposition}

\newtheorem{claim}[theorem]{Claim}
\theoremstyle{definition}

\theoremstyle{definition}
\newtheorem{remark}[theorem]{Remark}
\theoremstyle{definition}

\numberwithin{equation}{section}
\numberwithin{table}{section}
\numberwithin{figure}{section}

\title{Resonant semilinear Robin problems with a general potential}
\author[1]{\textbf{N.S. Papageorgiou}}
\author[2, 3]{\textbf{V.D. R\u{a}dulescu}}
\author[4]{\textbf{D.D. Repov\v{s}}}

\affil[1]{National Technical University, Department of Mathematics,
				Zografou Campus, Athens 15780, Greece}
\affil[2]{Department of Mathematics, Faculty of Sciences,
King Abdulaziz University, P.O. Box 80203, Jeddah 21589, Saudi Arabia}
\affil[3]{Department of Mathematics, University of Craiova, Craiova 200585, Romania}
\affil[4]{Faculty of Education and Faculty of Mathematics and Physics, University of Ljubljana, Ljubljana 1000, Slovenia}

\newcommand{\RR}{\mathbb R}
\newcommand{\NN}{\mathbb N}

\makeatletter
\renewcommand{\maketitle}{\bgroup\setlength{\parindent}{0pt}

\vspace{1truecm}
\begin{center}{\vbox{\titlefont\@title}}\end{center}
\vspace{0.5truecm}
\begin{center}{\@author} \end{center}

\egroup
}

\renewcommand{\@fnsymbol}[1]{%
    \ifcase#1 \or {\,\Letter\!} \or\textasteriskcentered\or \textasteriskcentered\textasteriskcentered
    \else\@ctrerr\fi}
\makeatother

\makeatletter
\newcommand{\hbibitem}[4]{\bibitem{#1}{#2}
\def\@tempa{#3}%
\def\@tempb{#4}%
\ifx\@tempa\@empty\ifx\@tempb\@empty{}{}\else{}{\href{https://doi.org/#4}{url}}\fi\else
{\href{http://www.ams.org/mathscinet-getitem?mr=#3}{MR#3}}\ifx\@tempb\@empty{}\else{; \href{https://doi.org/#4}{url}}\fi\fi}
\makeatother

\newcommand*{\titlefont}{\fontsize{18}{21.6}\selectfont\textbf}

\makeatletter
\renewcommand\@author{\ifx\AB@affillist\AB@empty\AB@author\else
      \ifnum\value{affil}>\value{Maxaffil}\def\rlap##1{##1}%
    \AB@authlist\\[\affilsep]\vbox{\AB@affillist}
    \else  \AB@authors\fi\fi}
\makeatother

\begin{document}

\maketitle

\pagestyle{plain}

\begin{center}
\noindent
\begin{minipage}{0.85\textwidth}\parindent=15.5pt

\bigskip

{\small{
\noindent {\bf Abstract.} We consider a semilinear Robin problem driven by the Laplacian plus an indefinite  and unbounded potential. The reaction term is a Carath\'eodory function which is resonant with respect to any nonprincipal eigenvalue both at $\pm \infty$ and 0. Using a variant of the reduction method, we show that the problem has at least two nontrivial smooth solutions.}
\smallskip

\smallskip

}

\end{minipage}
\end{center}

\section{Introduction}

Let $\Omega\subseteq\RR^N$ be a bounded domain with a $C^2$-boundary $\partial\Omega$. In this paper we study the following semilinear Robin problem:

\begin{equation}\label{eq1}
	\left\{\begin{array}{l}
		-\Delta u(z)+\xi(z)u(z) = f(z,u(z)) \ \mbox{in}\ \Omega, \ \\
		\displaystyle\frac{\partial u}{\partial n} + \beta(z)u = 0 \ \mbox{in}\ \partial\Omega\,.
	\end{array}\right\}
\end{equation}

In this problem, the potential function $\xi(\cdot)$ is unbounded and indefinite (that is, sign-changing). So, in problem (\ref{eq1}) the differential operator (on the left-hand side of the equation), is not coercive. The reaction term $f(z,x)$ is a Carath\'eodory function (that is, for all $x\in\RR$, $z\mapsto f(z,x)$ is measurable and for almost all $z\in\Omega$, $x\mapsto f(z,x)$ is continuous) and $f(z,\cdot)$ exhibits linear growth as $x\rightarrow\pm\infty$. In fact, we can have resonance with respect to any nonprincipal eigenvalue of $-\Delta+\xi(z)I$ with the Robin boundary condition. This general structure of the reaction term, makes the use of variational methods problematic. To overcome these difficulties, we develop a variant of the so-called ``reduction method", originally due to Amann \cite{1} and Castro \& Lazer \cite{3}. However, in contrast to the aforementioned works, the particular features of our problem lead to a reduction on an infinite dimensional subspace and this
 is a source of additional technical difficulties. In the boundary condition, $\frac{\partial u}{\partial n}$ is the normal derivative defined by extension of the continuous linear map
$$u\mapsto\frac{\partial u}{\partial n}=(Du,n)_{\RR^N}\ \mbox{for all}\ u\in C^1(\overline{\Omega}),$$
with $n(\cdot)$ being the outward unit normal on $\partial\Omega$. The boundary coefficient $\beta\in W^{1,\infty}(\partial\Omega)$ satisfies $\beta(z)\geq 0$ for all $z\in\partial\Omega$. We can have $\beta\equiv 0$, which corresponds to the Neumann problem.

Recently there have been existence and multiplicity results for semilinear elliptic problems with general potential. We mention the works of Hu \& Papageorgiou \cite{9}, Kyritsi \& Papageorgiou \cite{10}, Papageorgiou \& Papalini \cite{12}, Qin, Tang \& Zhang \cite{17} (Dirichlet problems), Gasinski \& Papageorgiou \cite{6}, Papageorgiou \& R\u adulescu \cite{13, 14} (Neumann problems) and for Robin problems there are the works of Shi \& Li \cite{18} (superlinear reaction), D'Agui, Marano \& Papageorgiou \cite{4} (asymmetric reaction), Hu \& Papageorgiou (logistic reaction) and Papageorgiou \& R\u adulescu \cite{15} (reaction with zeros). In all the aforementioned works the conditions are in many respects more restrictive or different and consequently the mathematical tools are different. It seems that our work here is the first to use this variant of the reduction method on Robin problems.

\section{Mathematical background}

Let $X$ be a Banach space and let $X^*$ be its topological dual. By $\left\langle \cdot,\cdot\right\rangle$ we denote the duality brackets for the pair $(X^*,X)$. Given $\varphi\in C^1(X,\RR)$, we say that $\varphi$ satisfies the ``Cerami condition" (the ``C-condition" for short), if the following property holds
\begin{center}
``Every sequence $\{u_n\}_{n\geq 1}\subseteq X$ such that $\{\varphi(u_n)\}_{n\geq 1}\subseteq\RR$ is bounded and
$$(1+||u_n||)\varphi'(u_n)\rightarrow 0\ \mbox{in}\ X^*,$$
admits a strongly convergent subsequence".
\end{center}

This is a compactness-type condition on the functional $\varphi$ and is more general that the usual Palais-Smale condition. The two notions are equivalent when $\varphi$ is bounded below (see Motreanu, Motreanu \& Papageorgiou \cite[p. 104]{11}).

Our multiplicity result will use the following abstract ``local linking" theorem of Brezis \& Nirenberg \cite{2}.
\begin{theorem}\label{th1}
	Let $X$ be a Banach space such that $X=Y\oplus V$ with ${\rm dim}\, Y<+\infty$. Assume that $\varphi\in C^1(X,\RR)$ satisfies the C-condition, it is bounded below, $\varphi(0)=0$, $\inf\limits_{X}\varphi<0$ and there exists $\rho>0$ such that
	\begin{eqnarray*}
		&&\varphi(y)\leq 0\ \mbox{for all}\ y\in Y\ \mbox{with}\ ||y||\leq\rho,\\
		&&\varphi(v)\geq 0\ \mbox{for all}\ v\in V\ \mbox{with}\ ||v||\leq\rho
	\end{eqnarray*}
	(that is, $\varphi$ has a local linking at $u=0$ with respect to the direct sum $Y\oplus V$). Then $\varphi$ has at least two nontrivial critical points.
\end{theorem}
\begin{remark}
	The result is true even if one of the component subspaces $Y$ or $V$ is trivial. Moreover, if ${\rm dim}\,V=0$, then we can allow $Y$ to be infinite dimensional.
\end{remark}
	
	We will use the following spaces:
	\begin{itemize}
		\item the Sobolev space $H^1(\Omega)$;
		\item the Banach space $C^1(\overline{\Omega})$; and
		\item the ``boundary" Lebesgue spaces $L^r(\partial\Omega)\ 1\leq r\leq\infty$.
	\end{itemize}
	
	The Sobolev space $H^1(\Omega)$ is a Hilbert space with the following inner product
	$$(u,v)=\int_{\Omega}uvdz+\int_{\Omega}(Du,Dv)_{\RR^N}dz\ \mbox{for all}\ u,v\in H^1(\Omega).$$
	
	By $||\cdot||$ we denote the norm corresponding to this inner product, that is,
	$$||u||=[||u||^2_2+||Du||_2^2]^{1/2}\ \mbox{for all}\ u,v\in H^1(\Omega).$$
	
	On $\partial\Omega$ we consider the $(N-1)$-dimensional Hausdorff (surface) measure denoted by $\sigma(\cdot)$. Using this measure on $\partial\Omega$, we can define in the usual way the Lebesgue spaces $L^r(\partial\Omega)$, $1\leq r\leq\infty$. From the theory of Sobolev spaces we know that there exists a unique continuous linear map $\gamma_0:H^1(\Omega)\rightarrow L^2(\partial\Omega)$, known as the ``trace map", which satisfies
	$$\gamma_0(u)=u|_{\partial\Omega}\ \mbox{for all}\ u\in H^1(\Omega)\cap C(\overline{\Omega}).$$

	So, the trace map assigns ``boundary values" to any Sobolev function (not just to the regular ones). This map is compact into $L^r(\partial\Omega)$ for all $r\in\left[1,\frac{2(N-1)}{N-2}\right)$ if $N\geq 3$ and into $L^r(\partial\Omega)$ for all $r\geq 1$ if $N=1,2$. Also, we have
	$${\rm ker}\, \gamma_0=H^1_0(\Omega)\ \mbox{and}\ {\rm im}\, \gamma_0=H^{\frac{1}{2},2}(\partial\Omega).$$
	
	In what follows, for the sake of notational simplicity, we shall drop the use the trace map $\gamma_0$. The restrictions of all Sobolev functions on $\partial\Omega$, are understood in the sense of traces.
	
	Next, we recall some basic facts about the spectrum of the differential operator $-\Delta+\xi(z)I$ with the Robin boundary condition. So, we consider the following linear eigenvalue problem:
	\begin{eqnarray}\label{eq2}
		\left\{\begin{array}{ll}
			-\Delta u(z)+\xi(z)u(z)=\hat{\lambda}u(z)&\mbox{in}\ \Omega,\\
			\displaystyle\frac{\partial u}{\partial n}+\beta(z)u=0&\mbox{on}\ \partial\Omega
		\end{array}\right\}
	\end{eqnarray}

Our conditions on the data of (\ref{eq2}) are the following:
\begin{itemize}
	\item	$\xi\in L^s(\Omega)$ with $s>N$; and
	\item $\beta\in W^{1,\infty}(\partial\Omega)$ with $\beta(z)\geq 0$ for all $z\in\partial\Omega$.
\end{itemize}

Let $\gamma:H^1(\Omega)\rightarrow \RR$ be the $C^1$-functional defined by
$$\gamma(u)=||Du||_2^2+\int_{\Omega}\xi(z)u^2dz+\int_{\partial\Omega}\beta(z)u^2d\sigma\ \mbox{for all}\ u\in H^1(\Omega).$$

By D'Agui, Marano \& Papageorgiou \cite{4}, we know that there exists $\mu>0$ such that
\begin{equation}\label{eq3}
	\gamma(u)+\mu||u||^2_2\geq c_0||u||^2\ \mbox{for all}\ u\in H^1(\Omega),\ \mbox{and some}\ c_0>0.
\end{equation}

Using (\ref{eq3}) and the spectral theorem for compact self-adjoint operators on a Hilbert space, we produce the spectrum $\sigma_0(\xi)$ of (\ref{eq2}) and we have that $\sigma_0(\xi)=\{\hat{\lambda}_k\}_{k\geq 1}$ a sequence of distinct eigenvalues with $\hat{\lambda}_k\rightarrow+\infty$ as $k\rightarrow+\infty$. By $E(\hat{\lambda}_k)$ (for all $k\in\NN$), we denote the eigenspace corresponding to the eigenvalue $\hat{\lambda}_k$. We know that $E(\hat{\lambda}_k)$ is finite dimensional. Moreover, the regularity theory of Wang \cite{19} implies that $E(\hat{\lambda}_k)\subseteq C^1(\overline{\Omega})$ for all $k\in\NN$. The Sobolev space $H^1(\Omega)$ admits the following orthogonal direct sum decomposition
$$H^1(\Omega)=\overline{{\underset{\mathrm{k\geq 1}}\oplus}E(\hat{\lambda}_k)}.$$

The elements of $\sigma_0(\xi)$ have the following properties:
\begin{itemize}
	\item $\hat{\lambda}_1$ is simple (that is, ${\rm dim}\, E(\hat{\lambda}_1)=1$).\\
	\begin{eqnarray}
		&&\bullet\ \ \hat{\lambda}_1=\inf\left[\frac{\gamma(u)}{||u||^2_2}:u\in H^1(\Omega),u\neq 0\right].\hspace{7cm}\label{eq4}\\
		&&\bullet\ \ \hat{\lambda}_m=\inf\left[\frac{\gamma(u)}{||u||^2_2}:u\in\overline{{\underset{\mathrm{k\geq m}}\oplus}E(\hat{\lambda}_k)},u\neq 0\right]\hspace{7cm}\nonumber\\
		&&\hspace{1cm}=\sup\left[\frac{\gamma(u)}{||u||^2_2}:u\in\overset{m}{\underset{\mathrm{k=1}}\oplus}E(\hat{\lambda}_k),u\neq 0\right]\ \mbox{for}\ m\geq 2.\label{eq5}
	\end{eqnarray}
\end{itemize}

The infimum in (\ref{eq4}) is realized on $E(\hat{\lambda}_1)$, while both the infimum and supremum in (\ref{eq5}) are realized on $E(\hat{\lambda}_m)$. It follows that the elements of $E(\hat{\lambda}_1)$ have fixed sign, while those of $E(\hat{\lambda}_m)$ ($m\geq 2$) are nodal (sign-changing). The eigenspaces have the so-called ``Unique Continuation Property" (UCP for short) which says that if $u\in E(\hat{\lambda}_k)$ and $u(\cdot)$ vanishes on a set of positive Lebesgue measure, then $u\equiv 0$. As a consequence of the UCP, we have the following useful inequalities (see D'Agui, Marano \& Papageorgiou \cite{4}).
\begin{lemma}\label{lem2}
	\begin{itemize}
		\item[(a)] If $\eta\in L^{\infty}(\Omega),\ \eta(z)\geq\hat{\lambda}_m$ for almost all $z\in\Omega,\ m\in\NN$ and $\eta\neq \hat{\lambda}_m$, then there exists $c_1>0$ such that
		$$\gamma(u)-\int_{\Omega}\eta(z)u^2dz\leq-c_1||u||^2\ \mbox{for all}\ u\in\overset{m}{\underset{\mathrm{k=1}}\oplus}E(\hat{\lambda}_k).$$
		\item[(b)] If $\eta\in L^{\infty}(\Omega),\ \eta(z)\leq\hat{\lambda}_m$ for almost all $z\in\Omega$, $m\in\NN$ and $\eta\neq\hat{\lambda}_m$ then there exists $c_2>0$ such that
		$$\gamma(u)-\int_{\Omega}\eta(z)u^2dz\geq c_2||u||^2\ \mbox{for all}\ u\in\overline{{\underset{\mathrm{k\geq m}}\oplus E(\hat{\lambda}_k)}}.$$
	\end{itemize}
\end{lemma}

Given $m\in\NN$, let $H_-=\overset{m}{\underset{\mathrm{k=1}}\oplus}E(\hat{\lambda}_k)$, $H^0=E(\hat{\lambda}_{m+1})$, $H_+=\overline{{\underset{\mathrm{k\geq m+2}}\oplus E(\hat{\lambda}_k)}}$. We have the following orthogonal direct sum decomposition
$$H^1(\Omega)=H_-\oplus H^0\oplus H_+.$$

So, every $u\in H^1(\Omega)$ admits a unique sum decomposition of the form
$$u=\bar{u}+u^0+\hat{u}\ \mbox{with}\ \bar{u}\in H_-,\ u^0\in H^0,\ \hat{u}\in H_+.$$

Also, we set
$$V=H^0\oplus H_+.$$

Finally, let us fix our notation. By $|\cdot|_N$ we denote the Lebesgue measure on $\RR^N$ and by $A\in\mathcal{L}(H^1(\Omega),H^1(\Omega)^*)$ the linear operator defined by
$$\left\langle A(u),h\right\rangle=\int_{\Omega}(Du,Dh)_{\RR^N}dz\ \mbox{for all}\ u,h\in H^1(\Omega)$$
(by $\left\langle \cdot,\cdot\right\rangle$ we denote the duality brackets for the pair $(H^1(\Omega)^*,H^1(\Omega))$). Also, given a measurable function $f:\Omega\times\RR\rightarrow\RR$ (for example a Carath\'eodory function), we set
$$N_f(u)(\cdot)=f(\cdot,u(\cdot))\ \mbox{for all}\ u\in H^1(\Omega)$$
(the Nemytski map corresponding to $f$). Evidently, $z\mapsto N_f(u)(z)$ is measurable. For $\varphi\in C^1(X,\RR)$, we set
$$K_{\varphi}=\{u\in X:\varphi'(u)=0\}$$
(the critical set of $\varphi$).

\section{Pair of nontrivial solutions}

The hypotheses on the data of (\ref{eq1}) are the following:

\begin{itemize}
	\item
$H(\xi)$: $\xi\in L^s(\Omega)$ with $s>N$; and

\item
$H(\beta)$: $\beta\in W^{1,\infty}(\partial\Omega)$ with $\beta(z)\geq 0$ for all $z\in\partial\Omega$.
\end{itemize}

\begin{remark}
	We can have $\beta\equiv 0$ and this case corresponds to the Neumann problem.
\end{remark}

$H(f)$: $f:\Omega\times\RR\rightarrow\RR$ is a Carath\'eodory function such that $f(z,0)=0$ for almost all $z\in\Omega$ and
\begin{itemize}
		\item[(i)] $|f(z,x)|\leq a(z)(1+|x|)$ for almost all $z\in\Omega$, and all $x\in\RR$ with $a\in L^{\infty}(\Omega)_+$;
		\item[(ii)] there exist $m\in\NN$ and $\eta\in L^{\infty}(\Omega)$ such that
		\begin{eqnarray*}
			&&\eta(z)\geq\hat{\lambda}_m\ \mbox{for almost all}\ z\in\Omega,\eta\not\equiv\hat{\lambda}_m,\\
			&&(f(z,x)-f(z,x'))(x-x')\geq\eta(z)(x-x')^2\ \mbox{for almost all}\ z\in\Omega,\ \mbox{and all}\ x,x'\in\RR;
		\end{eqnarray*}
		\item[(iii)] if $F(z,x)=\int^x_0f(z,s)ds$, then $\limsup\limits_{x\rightarrow\pm\infty}\frac{2F(z,x)}{x^2}\leq \hat{\lambda}_{m+1}$ and
		$$\lim\limits_{x\rightarrow\pm\infty}[f(z,x)x-2F(z,x)]=+\infty\ \mbox{uniformly for almost all}\ z\in\Omega;$$
		\item[(iv)] there exist $l\in\NN$, $l\geq m+2$, a function $\vartheta\in L^{\infty}(\Omega)$ and $\delta>0$ such that
		\begin{eqnarray*}
			&&\vartheta(z)\leq\hat{\lambda}_l\ \mbox{for almost all}\ z\in\Omega,\ \vartheta\not\equiv\hat{\lambda}_l,\\
			&&\hat{\lambda}_{l-1}x^2\leq f(z,x)x\leq\vartheta(z)x^2\ \mbox{for almost all}\ z\in\Omega,\ \mbox{and all}\ |x|\leq\delta.
		\end{eqnarray*}
\end{itemize}

Let $\varphi:H^1(\Omega)\rightarrow\RR$ be the energy (Euler) functional for problem (\ref{eq1}) defined by
$$\varphi(u)=\frac{1}{2}\gamma(u)-\int_{\Omega}F(z,u)dz\ \mbox{for all}\ u\in H^1(\Omega).$$

Evidently, $\varphi\in C^1(H^1(\Omega))$.

Recall that
$$H^1(\Omega)=H_-\oplus H^0\oplus H_+$$
with $H_-=\overset{m}{\underset{\mathrm{k=1}}\oplus}E(\hat{\lambda}_k).\ H^0=E(\hat{\lambda}_{m+1}),\ H_+=\overline{{\underset{\mathrm{k\geq m+2}}\oplus}E(\hat{\lambda}_k)}$ and
$$V=H^0\oplus H_+.$$

The next proposition is crucial in the implementation of the reduction method.
\begin{proposition}\label{prop3}
	If hypotheses $H(\xi),H(\beta),H(f)$ hold, then there exists a continuous map $\tau:V\rightarrow H_-$ such that
	$$\varphi(v+\tau(v))=\max[\varphi(v+y):y\in H_-]\ \mbox{for all}\ v\in V.$$
\end{proposition}
\begin{proof}
	We fix $v\in V$ and consider the $C^1$-functional $\varphi_v:H^1(\Omega)\rightarrow\RR$ defined by
	$$\varphi_v(u)=\varphi(v+u)\ \mbox{for all}\ u\in H^1(\Omega).$$
	
	By $i_{H_-}:H_-\rightarrow H^1(\Omega)$ we denote the embedding of $H_-$ into $H^1(\Omega)$. Let
	$$\hat{\varphi}_v=\varphi_v\circ i_{H_-}.$$
	
	From the chain rule, we have
	\begin{equation}\label{eq6}
		\hat{\varphi}'_v=p_{H^*_-}\circ\varphi'_v,
	\end{equation}
	with $p_{H^*_-}$ being the orthogonal projection of the Hilbert space $H^1(\Omega)$ onto $H^*_-$. By $\left\langle \cdot,\cdot\right\rangle_{H_-}$ we denote the duality brackets for the pair $(H^*_-,H_-)$. For $y,\ y'\in H_-$, we have
	\begin{eqnarray}\label{eq7}
		&&\left\langle \hat{\varphi}'_v(y)-\hat{\varphi}'_v(y'),y-y'\right\rangle_{H_-}\nonumber\\
		&=&\left\langle \varphi'_v(y)-\varphi'_v(y'),y-y'\right\rangle\ (\mbox{see (\ref{eq6})})\nonumber\\
		&=&\gamma(y-y')-\int_{\Omega}(f(z,v+y)-f(z,v+y'))(y-y')dz\nonumber\\
		&\leq&\gamma(y-y')-\int_{\Omega}\eta(z)(y-y')^2dz\ (\mbox{see hypothesis}\ H(f)(ii))\nonumber\\
		&\leq&-c_1||y-y'||^2\ (\mbox{see Lemma \ref{lem2}}).
	\end{eqnarray}
	
	This implies that
	\begin{equation}\label{eq8}
		-\hat{\varphi}'_v\ \mbox{is strongly monotone and therefore}\ -\hat{\varphi}_v\ \mbox{is strictly convex}.
	\end{equation}
	
	We have
	\begin{eqnarray}\label{eq9}
		\left\langle \hat{\varphi}'_v(y),y\right\rangle_{H_-}&=&\left\langle \varphi'_v(y),y\right\rangle\nonumber\\
		&=&\left\langle \varphi'_v(y)-\varphi'_v(0),y\right\rangle+\left\langle \varphi'_v(0),y\right\rangle\nonumber\\
		&\leq&-c_1||y||^2+c_3||y||\ \mbox{for some}\ c_3>0\ (\mbox{see (\ref{eq7})}),\nonumber\\
		\Rightarrow-\hat{\varphi}'_v\ \mbox{is coercive}&&
	\end{eqnarray}
	
	The continuity and monotonicity of $-\hat{\varphi}'_v$ (see (\ref{eq8})), imply that
	\begin{equation}\label{eq10}
		-\hat{\varphi}'_v\ \mbox{is maximal monotone}.
	\end{equation}
	
	However, a maximal monotone and coercive map is surjective (see, for example, Hu \& Papageorgiou \cite[p. 322]{8}). So,  we infer from (\ref{eq9}) and (\ref{eq10}) that there is a unique $y_0\in H_-$ such that
	\begin{equation}\label{eq11}
		\hat{\varphi}'_v(y_0)=0\ (\mbox{see (\ref{eq8})}).
	\end{equation}
	
	Moreover, $y_0$ is the unique maximizer of the function $\hat\varphi_v$. So, we can define the map $\tau:V\rightarrow H_-$ by setting $\tau(v)=y_0$. Then we have
	\begin{eqnarray}
		&&\varphi(v+\tau(v))=\max[\varphi(v+y):y\in H_-],\label{eq12}\\
		&\Rightarrow&p_{H^*_-}\varphi'(v+\tau(v))=0\ (\mbox{see (\ref{eq11}) and (\ref{eq6})}).\label{eq13}
	\end{eqnarray}
	
	We need to show that the map $\tau:V\rightarrow H_-$ is continuous. To this end, let $v_n\rightarrow v$ in $V$. First, note that if $\bar{u}\in H_-$, then
	\begin{eqnarray*}
		\varphi(\bar{u})&=&\frac{1}{2}\gamma(\bar{u})-\int_{\Omega}F(z,\bar{u})dz\\
		&\leq&\frac{1}{2}\gamma(\bar{u})-\frac{1}{2}\int_{\Omega}\eta(z)\bar{u}^2dz\ (\mbox{see hypothesis}\ H(f)(ii))\\
		&\leq&-c_1||\bar{u}||^2\ (\mbox{see Lemma \ref{lem2}}),\\
		\Rightarrow\tau(0)=0.
	\end{eqnarray*}
	
Since $\varphi\in C^1(H^1(\Omega))$ and $\varphi'(u)=\gamma'(u)-N_f(u)$, it follows that $\varphi'$ is bounded on bounded sets of $H^1(\Omega)$. Therefore
$$||\varphi'(v_n)||_*\leq c_4$$
with $c_4>0$ independent of $n\in\NN$ (recall that $v_n\rightarrow v$ in $H^1(\Omega)$).

	Then we have
	\begin{eqnarray*}
		0&=&\left\langle \varphi'(v_n+\tau(v_n)),\tau(v_n)\right\rangle\ (\mbox{see (\ref{eq13})})\\
		&=&\left\langle \varphi'(v_n+\tau(v_n))-\varphi'(v_n+\tau(0)),\tau(v_n)\right\rangle+\left\langle \varphi'(v_n+\tau(0)),\tau(v_n)\right\rangle\\
		&\leq&-c_1||\tau(v_n)||^2+c_4||\tau(v_n)||,\  \mbox{for all}\ n\in\NN\ (\mbox{see (\ref{eq7})})\\
		\Rightarrow&&\{\tau(v_n)\}_{n\geq 1}\subseteq H_-\ \mbox{is bounded}.
	\end{eqnarray*}
	
	By passing to a suitable subsequence if necessary and using the finite dimensionality of $H_-$, we can infer that
	\begin{equation}\label{eq14}
		\tau(v_n)\rightarrow\hat{y}\ \mbox{in}\ H^1(\Omega),\ \hat{y}\in H_-.
	\end{equation}
	
	We have
	\begin{eqnarray*}
		&&\varphi(v_n+\tau(v_n))\leq\varphi(v_n+y)\ \mbox{for all}\ y\in H_-,\ \mbox{all}\ n\in\NN\ (\mbox{see (\ref{eq12})}),\\
		&\Rightarrow&\varphi(v+\hat{y})\leq\varphi(v+y)\ \mbox{for all}\ y\in H_-\ (\mbox{see (\ref{eq14}) and recall that}\ \varphi\ \mbox{is continuous}),\\
		&\Rightarrow&\hat{y}=\tau(v).
	\end{eqnarray*}
	
	By the Urysohn convergence criterion (see, for example, Gasinski \& Papageorgiou \cite[p. 33]{7}),  we have for the original sequence
	\begin{eqnarray*}
		&&\tau(v_n)\rightarrow \tau(v)\ \mbox{in}\ H_-,\\
		&\Rightarrow&\tau(\cdot)\ \mbox{is continuous.}
	\end{eqnarray*}
\end{proof}

Consider the functional $\tilde{\varphi}:V\rightarrow\RR$ defined by
$$\tilde{\varphi}(v)=\varphi(v+\tau(v))\ \mbox{for all}\ v\in V.$$
\begin{proposition}\label{prop4}
	If hypotheses $H(\xi), H(\beta), H(f)$ hold, then $\tilde{\varphi}\in C^1(V,\RR)$ and $\tilde{\varphi}'(v)=p_{V^*}\varphi'(v+\tau(v))$ for all $v\in V$ (here $p_{V^*}$ denotes the orthogonal projection of the Hilbert space $H^1(\Omega)^*$ onto $V^*$).
\end{proposition}
\begin{proof}
	Let $v,h\in V$ and $t>0$. We have
	\begin{eqnarray}\label{eq15}
		&&\frac{1}{t}\left[\tilde{\varphi}(v+th)-\tilde{\varphi}(v)\right]\nonumber\\
		&\geq&\frac{1}{t}\left[\varphi(v+th+\tau(v))-\varphi(v+\tau(v))\right]\ (\mbox{see (\ref{eq12})}),\nonumber\\
		&\Rightarrow&\liminf\limits_{t\rightarrow 0^+}\frac{1}{t}\left[\tilde{\varphi}(v+th)-\tilde{\varphi}(v)\right]\geq\left\langle \varphi'(v+\tau(v)),h\right\rangle.
	\end{eqnarray}
	
	Also, we have
	\begin{eqnarray}\label{eq16}
		&&\frac{1}{t}\left[\tilde{\varphi}(v+th)-\tilde{\varphi}(v)\right]\nonumber\\
		&\leq&\frac{1}{t}\left[\varphi(v+th+\tau(v+th))-\varphi(v+\tau(v+th))\right]\nonumber\\
		&\Rightarrow&\limsup\limits_{t\rightarrow 0^+}\frac{1}{t}\left[\tilde{\varphi}(v+th)-\tilde{\varphi}(v)\right]\leq\left\langle \varphi'(v+\tau(v)),h\right\rangle\\
		&&(\mbox{recall that}\ \tau(\cdot)\ \mbox{is continuous, see Proposition \ref{prop3} and that}\ \varphi\in C^1(H^1(\Omega),\RR)).\nonumber
	\end{eqnarray}
	
	From (\ref{eq15}) and (\ref{eq16}) it follows that
	\begin{equation}\label{eq17}
		\lim\limits_{t\rightarrow 0^+}\frac{1}{t}\left[\tilde{\varphi}(v+th)-\tilde{\varphi}(v)\right]=\langle \varphi'(v+\tau(v)),h\rangle\ \mbox{for all}\ v,h\in V.
	\end{equation}
	
	Similarly we show that
	\begin{equation}\label{eq18}
		\lim\limits_{t\rightarrow 0^-}\frac{1}{t}\left[\tilde{\varphi}(v+th)-\tilde{\varphi}(v)\right]=\left\langle \varphi'(v+\tau(v)),h\right\rangle\ \mbox{for all}\ v,h\in V.
	\end{equation}
	
	From (\ref{eq17}) and (\ref{eq18}) we conclude that
	$$\tilde{\varphi}\in C^1(V,\RR)\ \mbox{and}\ \tilde{\varphi}'(v)=p_{V^*}\varphi'(v+\tau(v))\ \mbox{for all}\ v\in V.$$
\end{proof}
\begin{proposition}\label{prop5}
	If hypotheses $H(\xi),H(\beta),H(f)$ hold, then $v\in K_{\tilde{\varphi}}$ if and only if $v+\tau(v)\in K_{\varphi}$.
\end{proposition}
\begin{proof}
	$\Leftarrow$ Follows from Proposition \ref{prop4}.
	
	$\Rightarrow$ Let $v\in K_{\tilde{\varphi}}$. Then
	\begin{eqnarray}\label{eq19}
		&&0=\tilde{\varphi}'(v)=p_{V^*}\varphi'(v+\tau(v))\ (\mbox{see Proposition \ref{prop4}}),\nonumber\\
		&\Rightarrow&\varphi'(v+\tau(v))\in H^*_-\ (\mbox{recall that}\ H^1(\Omega)^*=H^*_-\oplus V^*).
	\end{eqnarray}
	
	On the other hand from (\ref{eq13}) we have
	\begin{eqnarray}\label{eq20}
		&&p_{H^*_-}\varphi'(v+\tau(v))=0,\nonumber\\
		&\Rightarrow&\varphi'(v+\tau(v))\in V^*.
	\end{eqnarray}
	
	But $H^*_-\cap V^*=\{0\}$. So,  it follows from (\ref{eq19}) and (\ref{eq20}) that
	\begin{eqnarray*}
		&&\varphi'(v+\tau(v))=0,\\
		&\Rightarrow&v+\tau(v)\in K_{\varphi}.
	\end{eqnarray*}
\end{proof}
\begin{proposition}\label{prop6}
	If hypotheses $H(\xi),H(\beta),H(f)$ hold, then $\tilde{\varphi}$ is coercive.
\end{proposition}
\begin{proof}
	Let $\psi=\varphi|_V$. Evidently, $\psi\in C^1(V,\RR)$ and by the chain rule we have
	\begin{equation}\label{eq21}
		\psi'=p_{V^*}\circ\varphi'.
	\end{equation}
	\begin{claim}\label{cl1}
		$\psi$ satisfies the C-condition.
	\end{claim}
	
	Let $\{v_n\}_{n\geq 1}\subseteq V$ be a sequence such that
	\begin{eqnarray}
		&&|\psi(v_n)|\leq M_1\ \mbox{for some}\ M_1>0,\ \mbox{and all}\ n\in\NN,\label{eq22}\\
		&&(1+||v_n||)\psi'(v_n)\rightarrow 0\ \mbox{in}\ V^*\ \mbox{as}\ n\rightarrow\infty\label{eq23}.
	\end{eqnarray}
	
	From (\ref{eq23}) we have
	\begin{eqnarray}\label{eq24}
		&&|\left\langle \psi'(v_n),h\right\rangle_V|\leq\frac{\epsilon_n||h||}{1+||v_n||}\ \mbox{for all}\ h\in V,\ n\in\NN,\ \mbox{with}\ \epsilon_n\rightarrow 0^+,\nonumber\\
		&\Rightarrow&|\left\langle \varphi'(v_n),h\right\rangle|\leq\frac{\epsilon_n||h||}{1+||v_n||}\ \mbox{for all}\ h\in V,\ n\in\NN\ (\mbox{see (\ref{eq21})}).
	\end{eqnarray}
	
	In (\ref{eq24}) we choose $h=v_n\in V$ and obtain
	\begin{equation}\label{eq25}
		\gamma(v_n)-\int_{\Omega}f(z,v_n)v_ndz\leq\epsilon_n\ \mbox{for all}\ n\in\NN.
	\end{equation}
	
	We show that $\{v_n\}_{n\geq 1}\subseteq V$ is bounded. Arguing by contradiction, suppose that
	\begin{equation}\label{eq26}
		||v_n||\rightarrow\infty\,.
	\end{equation}
	
	Let $\hat{w}_n=\frac{v_n}{||v_n||},\ n\in\NN$. Then $\hat{w}_n\in V,\ ||\hat{w}_n||=1$ for all $n\in\NN$. By passing to a suitable subsequence if necessary, we may assume that
	\begin{equation}\label{eq27}
		\hat{w}_n\stackrel{w}{\rightarrow}\hat{w}\ \mbox{in}\ H^1(\Omega)\ \mbox{and}\ \hat{w}_n\rightarrow\hat{w}\ \mbox{in}\ L^2(\Omega)\ \mbox{and}\ L^2(\partial\Omega).
	\end{equation}
	
	Hypotheses $H(f)$ imply that
	\begin{equation}\label{eq28}
		|f(z,x)|\leq c_5|x|\ \mbox{for almost all}\ z\in\Omega,\ \mbox{all}\ x\in\RR,\ \mbox{and some}\ c_5>0.
	\end{equation}
	
	By (\ref{eq24}) we have
	\begin{equation}\label{eq29}
		\left|\left\langle \gamma'(\hat{w}_n),h\right\rangle-\int_{\Omega}\frac{N_f(v_n)}{||v_n||}hdz\right|
\leq\frac{\epsilon_n||h||}{(1+||v_n||)||v_n||}\ \mbox{for all}\ n\in\NN,\ h\in H^1(\Omega).
	\end{equation}
	
	From (\ref{eq28}) and (\ref{eq27}) we see that
	\begin{equation}\label{eq30}
		\left\{\frac{N_f(v_n)}{||v_n||}\right\}_{n\geq 1}\subseteq L^2(\Omega)\ \mbox{is bounded}.
	\end{equation}
	
	So, if in (\ref{eq29}) we choose $h=\hat{w}_n-\hat{w}\in H^1(\Omega)$, pass to the limit as $n\rightarrow\infty$ and use (\ref{eq26}), (\ref{eq27}) and (\ref{eq30}), then
	\begin{eqnarray*}
		&&\lim\limits_{n\rightarrow\infty}\left\langle A(\hat{w}_n),\hat{w}_n-\hat{w}\right\rangle=0,\\
		&\Rightarrow&||D\hat{w}_n||_2\rightarrow||D\hat{w}||_2,\\
		&\Rightarrow&\hat{w}_n\rightarrow\hat{w}\ \mbox{in}\ H^1(\Omega)\\
		&&(\mbox{by the Kadec-Klee property, see Gasinski \& Papageorgiou \cite[p. 911]{5}}),\\
		&\Rightarrow&||\hat{w}||=1\ \mbox{and so}\ \hat{w}\neq 0.
	\end{eqnarray*}
	
	Let $\Omega_0=\{z\in\Omega:\hat{w}(z)\neq 0\}$. Then $|\Omega_0|_N>0$ and
	$$v_n(z)\rightarrow\pm\infty\ \mbox{for almost all}\ z\in\Omega_0\ (\mbox{see (\ref{eq26})}).$$
	
	Hypothesis $H(f)(iii)$ implies that
	\begin{equation}\label{eq31}
		f(z,v_n(z))v_n(z)-2F(z,v_n(z))\rightarrow+\infty\ \mbox{for almost all}\ z\in\Omega_0.
	\end{equation}
	
	From (\ref{eq31}) via Fatou's lemma (hypothesis $H(f)(iii)$ permits its use), we have
	\begin{equation}\label{eq32}
		\int_{\Omega_0}[f(z,v_n)v_n-2F(z,v_n)]dz\rightarrow+\infty.
	\end{equation}
	
	Using hypothesis $H(f)(iii)$  we see that we can find $M_2>0$ such that
	\begin{equation}\label{eq33}
		f(z,x)x-2F(z,x)\geq 0\ \mbox{for almost all}\ z\in\Omega,\ \mbox{all}\ |x|\geq M_2.
	\end{equation}
	
	So, we have
	\begin{eqnarray*}
		&&\int_{\Omega\backslash\Omega_0}[f(z,v_n)v_n-2F(z,v_n)]dz\\
		&=&\int_{(\Omega\backslash\Omega_0)\cap\{|v_n|\geq M_2\}}[f(z,v_n)v_n-2F(z,v_n)]dz+\\
		&&\hspace{1cm}\int_{(\Omega\backslash\Omega_0)\cap\{|v_n|<M_2\}}[f(z,v_n)v_n-2F(z,v_n)]dz\\
		&\geq&\int_{(\Omega\backslash\Omega_0)\cap\{|v_n|<M_2\}}[f(z,v_n)v_n-2F(z,v_n)]dz\ (\mbox{see (\ref{eq33})})\\
		&\geq&-M_3\ \mbox{for some}\ M_3>0,\ \mbox{all}\ n\in\NN\ (\mbox{see hypothesis}\ H(f)(i)).
	\end{eqnarray*}
	
	Then
	\begin{eqnarray}\label{eq34}
		&&\int_{\Omega}[f(z,v_n)v_n-2F(z,v_n)]dz\nonumber\\
		&=&\int_{\Omega_0}[f(z,v_n)v_n-2F(z,v_n)]dz+\int_{\Omega\backslash\Omega_0}[f(z,v_n)v_n-2F(z,v_n)]dz\nonumber\\
		&\geq&\int_{\Omega_0}[f(z,v_n)v_n-2F(z,v_n)]dz-M_3\ \mbox{for all}\ n\in\NN\nonumber\\
		\Rightarrow&&\int_{\Omega}[f(z,v_n)v_n-2F(z,v_n)]dz\rightarrow+\infty\ \mbox{as}\ n\rightarrow\infty\ (\mbox{see (\ref{eq32})}).
	\end{eqnarray}
	
	From (\ref{eq24}) with $h=v_n\in H^1(\Omega)$, we have
	\begin{equation}\label{eq35}
		-\gamma(v_n)+\int_{\Omega}f(z,v_n)v_ndz\leq\epsilon_n\ \mbox{for all}\ n\in\NN.
	\end{equation}
	
	Also, from (\ref{eq22}) we have
	\begin{equation}\label{eq36}
		\gamma(v_n)-\int_{\Omega}2F(z,v_n)dz\leq 2M_1\ \mbox{for all}\ n\in\NN.
	\end{equation}
	
	We add (\ref{eq35}) and (\ref{eq36}) and obtain
	\begin{equation}\label{eq37}
		\int_{\Omega}[f(z,v_n)v_n-2F(z,v_n)]dz\leq M_4\ \mbox{for some}\ M_4>0,\ \mbox{and all}\ n\in\NN.
	\end{equation}
	
	Comparing (\ref{eq34}) and (\ref{eq37}), we get a contradiction. This proves that $\{v_n\}_{n\geq 1}\subseteq V$ is bounded. So, we may assume that
	\begin{equation}\label{eq38}
		v_n\stackrel{w}{\rightarrow}u\ \mbox{in}\ H^1(\Omega)\ \mbox{and}\ v_n\rightarrow u\ \mbox{in}\ L^2(\Omega)\ \mbox{and}\ L^2(\partial\Omega).
	\end{equation}
	
	In (\ref{eq24}) we choose $h=v_n-u\in H^1(\Omega)$, pass to the limit as $n\rightarrow\infty$ and use (\ref{eq38}). Then
	\begin{eqnarray*}
		&&\lim\limits_{n\rightarrow\infty}\left\langle A(v_n),v_n-u\right\rangle=0,\\
		&\Rightarrow& v_n\rightarrow u\ \mbox{in}\ H^1(\Omega)\ (\mbox{as before via the Kadec-Klee property}).
	\end{eqnarray*}
	
	This proves Claim \ref{cl1}.
	\begin{claim}\label{cl2}
		$\hat{\lambda}_{m+1}x^2-2F(z,x)\rightarrow+\infty$ as $x\rightarrow +\infty$ uniformly for almost all $z\in\Omega$.
	\end{claim}
	
	Hypothesis $H(f)(iii)$ implies that given any $\lambda>0$, we can find $M_5=M_5(\lambda)>0$ such that
	\begin{equation}\label{eq39}
		f(z,x)x-2F(z,x)\geq\lambda\ \mbox{for almost all}\ z\in\Omega,\ \mbox{and all}\ |x|\geq M_5.
	\end{equation}
	
	For almost all $z\in\Omega$, we have
	\begin{eqnarray}\label{eq40}
		&&\frac{d}{dx}\left(\frac{F(z,x)}{|x|^2}\right)=\frac{f(z,x)x-2F(z,x)}{|x|^2x}\left\{\begin{array}{ll}
			\geq\frac{\lambda}{x^2}&\mbox{if}\ x\geq M_5\\
			\leq\frac{\lambda}{|x|^2x}&\mbox{if}\ x\leq -M_5
		\end{array}\right.\ (\mbox{see (\ref{eq39})}),\nonumber\\
		&\Rightarrow&\frac{F(z,y)}{|y|^2}-\frac{F(z,v)}{|v|^2}\geq\frac{\lambda}{2}\left[\frac{1}{|v|^2}-\frac{1}{|y|^2}\right]\ \mbox{for all}\ |y|\geq|v|\geq M_5.
	\end{eqnarray}
	
	We let $|y|\rightarrow\infty$ and use hypothesis $H(f)(iii)$. Then
	$$\hat{\lambda}_{m+1}|v|^2-2F(z,v)\geq\lambda\ \mbox{for almost all}\ z\in\Omega,\ \mbox{and all}\ |v|\geq M_5.$$
	
	Since $\lambda>0$ is arbitrary, we conclude that
	$$\hat{\lambda}_{m+1}|v|^2-2F(z,v)\rightarrow+\infty\ \mbox{as $v\rightarrow +\infty$ uniformly for almost all}\ z\in\Omega.$$
	
	This proves Claim \ref{cl2}.
	
	For every $v\in V$, we have
	\begin{eqnarray}\label{eq41}
		\psi(v)=\varphi(v)&=&\frac{1}{2}\gamma(v)-\int_{\Omega}F(z,v)dz\nonumber\\
		&\geq&\int_{\Omega}\left[\frac{1}{2}\hat{\lambda}_{m+1}v^2-F(z,v)\right]dz\ (\mbox{see (\ref{eq5})})\nonumber\\
		&\geq&-c_6\ \mbox{for some}\ c_6>0\ (\mbox{see Claim \ref{cl2} and hypothesis H(f)(i)})\nonumber\\
		&\Rightarrow&\psi\ \mbox{is bounded below}.
	\end{eqnarray}
	
	From (\ref{eq41}) and Claim \ref{cl1} it follows that
	$$\psi\ \mbox{is coercive}$$
	(see Motreanu, Motreanu \& Papageorgiou \cite[p. 103]{11}).
	
	We have
	\begin{eqnarray*}
		&&\psi(v)=\varphi(v)\leq\varphi(v+\tau(v))=\tilde{\varphi}(v)\ \mbox{for all}\ v\in V\ (\mbox{see (\ref{eq12})}),\\
		&\Rightarrow&\tilde{\varphi}\ \mbox{is coercive}.
	\end{eqnarray*}
\end{proof}

From Proposition \ref{prop5}, we deduce that:
\begin{corollary}\label{cor6}
	If hypotheses $H(\xi),H(\beta),H(f)$ hold, then $\tilde{\varphi}$ is bounded below and satisfies the C-condition.
\end{corollary}

Next we show that $\tilde{\varphi}$ admits a local linking (see Theorem \ref{th1}) with respect to the orthogonal direct sum decomposition $V=W\oplus\hat{E}$ where $W=\overset{l-1}{\underset{\mathrm{i=m+1}}\oplus}E(\hat{\lambda}_i),\hat{E}=\overline{{\underset{\mathrm{i\geq l}}\oplus}E(\lambda_i)}$.
\begin{proposition}\label{prop7}
	If hypotheses $H(\xi),H(\beta),H(f)$ hold, then $\tilde{\varphi}$ has a local linking at $u=0$ with respect to the decomposition
	$$V=W\oplus\hat{E}.$$
\end{proposition}
\begin{proof}
	From hypotheses $H(f)(i),(iv)$, we see that given $r>2$, we can find $c_7=c_7(r)>0$ such that
	\begin{equation}\label{eq42}
		F(z,x)\leq\frac{\vartheta(z)}{2}x^2+c_7|x|^r\ \mbox{for almost all}\ z\in\Omega,\ \mbox{all}\ x\in\RR.
	\end{equation}
	
	For $\hat{v}\in\hat{E}$ we have
	\begin{eqnarray*}
		\tilde{\varphi}(\hat{v})&=&\varphi(\hat{v}+\tau(\hat{v}))\\
		&\geq&\varphi(\hat{v})\ \mbox{(see Proposition \ref{prop3})}\\
		&=&\frac{1}{2}\gamma(\hat{v})-\int_{\Omega}F(z,\hat{v})dz\\
		&\geq&\frac{1}{2}\gamma(\hat{v})-\frac{1}{2}\int_{\Omega}\vartheta(z)\hat{v}^2dz-c_8||\hat{v}||^r\ \mbox{for some}\ c_8>0\ (\mbox{see (\ref{eq42})})\\
		&\geq&c_9||\hat{v}||^2-c_8||\hat{v}||^r\ \mbox{for some}\ c_9>0\ (\mbox{see Lemma \ref{lem2}(b)}).
	\end{eqnarray*}
	
	Since $r>2$, we see that we can find $\rho_1\in(0,1)$ small such that
	\begin{equation}\label{eq43}
		\tilde{\varphi}(\hat{v})>0\ \mbox{for all}\ \hat{v}\in\hat{E}\ \mbox{with}\ 0<||\hat{v}||\leq\rho_1.
	\end{equation}
	
	The space $Y=H_-\oplus W$ is finite dimensional and so all norms are equivalent. Hence we can find $\epsilon_0>0$ such that
	\begin{equation}\label{eq44}
		y\in Y\ \mbox{and}\ ||y||\leq\epsilon_0\Rightarrow|y(z)|\leq\delta\ \mbox{for all}\ z\in\overline{\Omega}\ (\mbox{recall that}\ Y\subseteq C^1(\overline{\Omega})).
	\end{equation}
	
	By Proposition \ref{prop3} we know that $\tau(\cdot)$ is continuous. So, we can find $\rho_2>0$ such that
	\begin{equation}\label{eq45}
		\tilde{u}\in W\ \mbox{and}\ ||\tilde{u}||\leq\rho_2\Rightarrow||\tilde{u}+\tau(\tilde{u})||\leq\epsilon_0.
	\end{equation}
	
	From (\ref{eq44}) and (\ref{eq45}) it follows that
	\begin{eqnarray*}
		\tilde{\varphi}(\tilde{u})&=&\varphi(\tilde{u}+\tau(\tilde{u}))\\
		&=&\frac{1}{2}\gamma(\tilde{u}+\tau(\tilde{u}))-\int_{\Omega}F(z,\tilde{u}+\tau(\tilde{u}))dz\\
		&\leq&\frac{1}{2}\hat{\lambda}_{l-1}||\tilde{u}+\tau(\tilde{u})||^2_2-\frac{1}{2}\hat{\lambda}_{l-1}||\tilde{u}+\tau(\tilde{u})||^2_2\ (\mbox{see hypothesis}\ H(f)(iv))\\
		&=&0.
	\end{eqnarray*}
	
	So, we have that
	\begin{equation}\label{eq46}
		\tilde{\varphi}(\tilde{u})\leq 0\ \mbox{for all}\ \tilde{u}\in W\ \mbox{with}\ ||\tilde{u}||\leq\rho_2.
	\end{equation}
	
	If $\rho=\min\{\rho_1,\rho_2\}$, then from (\ref{eq43}) and (\ref{eq46}) we conclude that $\varphi$ has a local linking at $u=0$ with respect to the decomposition $V=W\oplus\hat{E}$.
\end{proof}

Now we are ready for proving our multiplicity theorem.
\begin{theorem}\label{th8}
	If hypotheses $H(\xi),H(\beta),H(f)$ hold, then problem (\ref{eq1}) admits at least two nontrivial solutions
	$$u_0,\hat{u}\in C^1(\overline{\Omega}).$$
\end{theorem}
\begin{proof}
	From Proposition \ref{prop7} we know that
	$$\inf\limits_{V}\tilde{\varphi}\leq 0.$$
	
	If $\inf\limits_{V}\tilde{\varphi}=0$, then by Proposition \ref{prop7} all $\tilde{u}\in W$ with $0<||\tilde{u}||\leq\rho$ are nontrivial critical points of $\tilde{\varphi}$. Hence $\tilde{u}+\tau(\tilde{u})$ are nontrivial critical points of $\varphi$ (see Proposition \ref{prop5}).
	
	If $\inf\limits_{V}\tilde{\varphi}<0$, then we can apply Theorem \ref{th1} (see Corollary \ref{cor6}) and produce two nontrivial critical points $\tilde{u}_0$ and $\tilde{u}_*$ of $\tilde{\varphi}$. Then
	$$u_0=\tilde{u}_0+\tau(\tilde{u}_0)\ \mbox{and}\ \hat{u}=\tilde{u}_*+\tau(\hat{u}_*)$$
	are two nontrivial critical points of $\varphi$ (see Proposition \ref{prop5}).
	
	For $u=u_0$ or $u=\hat{u}$, we have
	\begin{eqnarray}\label{eq47}
		&&-\Delta u(z)+\xi(z)u(z)=f(z,u(z))\ \mbox{for almost all}\ z\in\Omega,\\
		&&\frac{\partial u}{\partial n}+\beta(z)u=0\ \mbox{on}\ \partial\Omega\ (\mbox{see Papageorgiou \& R\u adulescu \cite{16,15}}).\nonumber
	\end{eqnarray}
	
	Evidently, hypotheses $H(f)$ imply that
	\begin{equation}\label{eq48}
		|f(z,x)|\leq c_{10}|x|\ \mbox{for almost all}\ x\in\RR,\ \mbox{and some}\ c_{10}>0.
	\end{equation}
	
	We set
	$$b(z)=\left\{\begin{array}{ll}
		\frac{f(z,u(z))}{u(z)}&\mbox{if}\ u(z)\neq 0\\
		0&\mbox{if}\ u(z)=0.
	\end{array}\right.$$
	
	From (\ref{eq48}) it follows that $b\in L^{\infty}(\Omega)$. From (\ref{eq47}) we have
	$$\left\{\begin{array}{ll}
		-\Delta u(z)=(b-\xi)(z)u(z)\ \mbox{for almost all}\ z\in\Omega,&\\
		\frac{\partial u}{\partial n}+\beta(z)u=0\ \mbox{on}\ \partial\Omega.&
	\end{array}\right.$$
	
	Note that $b-\xi\in L^s(\Omega),\ s>N$ (see hypothesis $H(\xi)$). Then Lemmata 5.1 and 5.2 of Wang \cite{19} imply that
	\begin{eqnarray*}
		&&u\in W^{2,s}(\Omega),\\
		&\Rightarrow&u\in C^{1,\alpha}(\overline{\Omega})\ \mbox{with}\ \alpha=1-\frac{N}{s}>0\ (\mbox{by the Sobolev embedding theorem}).
	\end{eqnarray*}
	
	Therefore $u_0,\hat{u}\in C^1(\overline{\Omega})$.
\end{proof}

\medskip
{\bf Acknowledgments.} This research was supported by the Slovenian Research Agency grants P1-0292, J1-8131, J1-7025. V.D. R\u adulescu acknowledges the support through a grant of the Romanian Ministry of Research and
Innovation, CNCS - UEFISCDI, project number PN-III-P4-ID-PCE-2016-0130, within
PNCDI III.

\end{document}